\documentclass[a4paper, 12pt]{amsart}

\usepackage{amssymb}
\usepackage{amsmath}
\usepackage{amscd}
\usepackage{amsthm}
\usepackage[centertags]{amsmath}
\usepackage{amsfonts}
\usepackage{newlfont}
\usepackage[all]{xy}
\usepackage{graphicx}
\usepackage[usenames]{color}
\usepackage{amsfonts, amssymb}
\usepackage{mathrsfs}
\usepackage{latexsym}

\newtheorem{theorem}{Theorem}[section]

\newtheorem{lemma}[theorem]{Lemma}

\theoremstyle{definition}
\newtheorem{remark}[theorem]{Remark}

\theoremstyle{definition}

\theoremstyle{remark}
\def\X{\mathcal{X}}

\newcommand{\kkk}{\mathbf k}
\newcommand{\PP}{\mathcal P}
\newcommand{\zC}{\mathbb C}

\begin{document}
\baselineskip=.75cm

\title{Lower bounds for norms of products of polynomials
on $L_p$ spaces}

\thanks{This project was supported in part by UBACyT W746  and CONICET PIP 0624. The third author has a doctoral fellowship from CONICET}

\author{Daniel Carando}
\author{Dami\'an Pinasco}
\author{Tom\'as Rodr\'iguez}

\address{Departamento de Matem\'{a}tica - Pab I,
Facultad de Cs. Exactas y Naturales, Universidad de Buenos Aires,
(1428) Buenos Aires, Argentina and IMAS-CONICET}
\email{dcarando@dm.uba.ar}
\email{jtrodrig@dm.uba.ar}
\address{Departamento de Matem\'{a}ticas y Estad\'{i}sticas, Universidad Torcuato Di Tella, Mi\~{n}ones 2177 (C1428ATG) Buenos Aires, Argentina and CONICET}
\email{dpinasco@utdt.edu}

\begin{abstract}
For $1< p <2$ we obtain sharp inequalities for the supremum of products of homogeneous polynomials on $L_p(\mu)$, whenever the number of factors is no greater than the dimension of these Banach spaces (a condition readily satisfied in the infinite dimensional settings). The results also holds for the Schatten classes $\mathcal S_p$. For $p>2$ we present some estimates on the involved constants.
\end{abstract}

%

\maketitle

\section*{Introduction}

This work is framed in what is sometimes called the \emph{factor problem} for homogeneous polynomials. Given homogeneous polynomials $P_1, \ldots, P_n$  defined on $(\zC^N,\|\cdot\|_p)$, our aim is to find the best constant $M$ such that
\begin{equation}\label{problema}
\Vert P_1 \cdots P_n\Vert \ge M \, \Vert P_1 \Vert \cdots \Vert P_n \Vert .
\end{equation}
The constant will necessarily depend on $p$ and on the degrees of the polynomials, but not on the number of variables $N$. And, of course, we must set what the norm of a polynomial is.

Recall that a mapping $P:X\to \zC$ is a (continuous) $k$-homogeneous  polynomial if there exists a (continuous) $k$-linear map $T:X\times\cdots \times X\to \zC$ such that $P(x)=T(x,\dots,x)$ for all $x\in X$.
The space of continuous $k-$homogeneous polynomials on a Banach space $X$ is denoted by $\mathcal P(^kX)$. It is a Banach space under the uniform norm $$\Vert P \Vert_{\mathcal P(^kX)}=\sup_{\Vert z \Vert_X=1} \vert P(z)\vert.$$ Considering this norm, inequality (\ref{problema}) was studied for polynomials defined on finite and infinite dimensional Banach spaces. For instance R. Ryan and B.~Turett \cite{RT} gave bounds for the special case where the polynomials $\{P_i\}_{i=1}^n$ are actually continuous linear forms on $X$. Moreover, C. Ben\'{i}tez, Y. Sarantopoulos and A. Tonge  \cite{BST} proved that if $P_i$ has degree $k_i$ for $1\le i \le n$, then inequality (\ref{problema}) holds with constant
\[
M=\frac{k_1^{k_1} \cdots k_n^{k_n}}{(k_1+\cdots + k_n)^{(k_1+\cdots +k_n)}}
\]
for any complex Banach space. The authors also showed that this is the best universal constant, since there are polynomials on $\ell_1$ for which equality prevails. However, for many spaces it is possible to improve this bound. For instance, for complex Hilbert spaces, the second named author  proved in \cite{P} that the optimal constant is
\begin{equation}\label{ctel2}
M=\sqrt{\frac{k_1^{k_1} \cdots k_n^{k_n}}{(k_1+\cdots + k_n)^{(k_1+\cdots +k_n)}}}.
\end{equation}
In this work we establish the best constant for complex $L_p(\Omega,\mu)$ spaces whenever $1<p<2$. We show that in this case inequality (\ref{problema}) holds with constant
\begin{equation}\label{eq-la cte p}
M=\sqrt[p]{\frac{k_1^{k_1} \cdots k_n^{k_n}}{(k_1+\cdots + k_n)^{(k_1+\cdots +k_n)}}}.
\end{equation}
The constant is optimal provided the involved spaces have enough dimension at least $n$. This constant also works (and is optimal) for polynomials on the Schatten classes $\mathcal S_p$. For the remaining values of $p$, we obtain some estimates of the optimal constants.

\section{Main results}

We begin with some definitions. If $E$ and $F$ are isomorphic Banach spaces, their {\it Banach-Mazur distance} (see \cite[Chapter~1]{Pi86}, \cite{T}) is defined as
\[
 d(E,F)=\inf \{ \Vert u \Vert \, \Vert u^{-1} \Vert \mid u:E \to F \text{ isomorphism} \} \, .
\]
Given a Banach space $X$ and $n\in \mathbb N$, we define $$D_n(X):=\sup\{d(E,\ell_2^n):E\text{ subspace of } X\text{ with } \dim E=n\}.$$
From Corollary 5 in \cite{L}, we obtain
\begin{equation}\label{eq-lewis}
D_n(L_p(\Omega,\mu)) \leq n^{|1/p-1/2|},
\end{equation}
whenever $L_p(\Omega,\mu)$ has dimension at least $n$.

The proof of the following lemma is inspired by Proposition 1 in \cite{RS}.

\begin{lemma}\label{mismogrado}
Let $X$ be a Banach space and let $P_1,\dots,P_n: X \to \mathbb C$ be ho\-mo\-ge\-neous polynomials of degree $k_1,\ldots, k_n$ respectively. Then
\[
\|P_1\cdots P_n\|_{\PP(^{\kkk}X)} \ge \sqrt{\frac{\prod_{i=1}^{n} {k_i^{k_i}}}{\mathbf k^{\mathbf k}}} \ D_n(X)^{-\kkk}\ \|P_1\|_{\PP(^{k_1}X)}\cdots \|P_n\|_{\PP(^{k_n}X)},
\]
where $\kkk = \sum_{i=1}^n{k_i}$.
\end{lemma}
\begin{proof}
Given $\varepsilon>0$, we can take a set of norm one vectors $\{x_1, \ldots, x_n\} \subset X$ such that $\vert P_j(x_j)\vert > (1- \varepsilon)\, \Vert P_j\Vert_{\PP(^{k_j}X)}$, for $1\le j \le n$. Let $E\subset X$ be any $n-$dimensional subspace containing the subspace spanned by $\{x_1, \ldots, x_n\}$ and let $T:\ell_2^n\rightarrow E$ be a norm one isomorphism with $\|T^{-1}\|\le D_n(X)$. We have

\begin{eqnarray}
\|P_1\cdots P_n\|_{\PP(^{\kkk}X)} &\ge & \|P_1\cdots P_n\|_{\PP(^{\kkk}E )} \ge \|(P_1\circ T)\cdots (P_n\circ T) \|_{\PP(^{\kkk}\ell_2^n )} \nonumber \\
&\ge& \sqrt{\frac{k_1^{k_1}\cdots k_n^{k_n}}{\kkk^{\kkk}}} \, \|(P_1\circ T) \|_{\PP(^{k_1} \ell_2^n )} \cdots \|(P_n\circ T) \|_{\PP(^{k_n} \ell_2^n )} \label{eq-pinasco2} \\
&\ge& \sqrt{\frac{k_1^{k_1}\cdots k_n^{k_n}}{\kkk^{\kkk}}} \, \frac 1 {\|T^{-1}\|^{\kkk}} \, \|P_1 \|_{\PP(^{k_1}E )} \cdots \|P_n\|_{\PP(^{k_n}E )} \nonumber \\
&>& \sqrt{\frac{k_1^{k_1}\cdots k_n^{k_n}}{\kkk^{\kkk}}} \, D_n(X)^{-\kkk} (1-\varepsilon)^n \|P_1\|_{\PP(^{k_1}X)}\cdots \|P_n\|_{\PP(^{k_n}X)}, \nonumber
\end{eqnarray}
where \eqref{eq-pinasco2} follows from \eqref{ctel2}.
\end{proof}
\begin{remark}\label{mismogrado-remark} If we restrict ourselves to the spaces $L_p(\Omega,\mu)$ and polynomials \emph{with the same degree}, we can combine Lemma \ref{mismogrado} with Lewis' result \eqref{eq-lewis} to obtain \begin{equation}\label{eq-gradosigualesp<2}
\|P_1\cdots P_n\|_{\PP(^{kn}L_p(\Omega,\mu))} \ge \frac 1 {n^{nk/p}}\ \|P_1\|_{\PP(^kL_p(\Omega,\mu))}\cdots \|P_n\|_{\PP(^kL_p(\Omega,\mu))}
\end{equation}
for $1\le p\le 2$. For $2\le p\le\infty$ we have
\begin{equation}\label{eq-gradosigualesp>2}
\|P_1\cdots P_n\|_{\PP(^{kn}L_p(\Omega,\mu))} \ge \frac 1 {n^{nk/q}}\ \|P_1\|_{\PP(^kL_p(\Omega,\mu))}\cdots \|P_n\|_{\PP(^kL_p(\Omega,\mu))},\nonumber
\end{equation}
where $q$ is the conjugate exponent of $p$.
\end{remark}
Note that \eqref{eq-gradosigualesp<2} is precisely \eqref{problema} with the constant $M$ given in \eqref{eq-la cte p}.
In order to extend this result to a general case, where the polynomials have arbitrary degrees, it is convenient to consider another particular case. In the sequel we will say that $P, Q: \ell_p^N \to \zC$ {\it depend on different variables} if it is possible to find disjoint subsets $I, J \subset \{1, 2, \ldots, N\}$, such that $P(\sum_{i=1}^N a_i e_i)=P(\sum_{i \in I} a_i e_i)$ and $Q(\sum_{i=1}^N a_i e_i)=Q(\sum_{i \in J} a_i e_i)$, for all $\{a_i\}_{i=1}^N \subset \zC$.

For polynomials depending on different variables,  \eqref{problema} becomes an equality when $M$ is given by \eqref{eq-la cte p}, as the following lemma shows.

\begin{lemma}\label{distintasvars}
Let $P_1,\dots,P_n$ be homogeneous polynomials of degrees $k_1,\dots,k_n$ respectively, defined on $\ell_p^N$, depending on different variables. If $\mathbf k=k_1+\cdots+k_n$, then we have
\begin{equation} \|P_1\cdots P_n\|_{\PP(^{\mathbf k}\ell^N_p)} = \sqrt[p]{\frac{\prod_{i=1}^{n} {k_i^{k_i}}} {\mathbf k^{\mathbf k}}}\, \|P_1\|_{\PP(^{k_1}\ell_p^N)} \cdots \|P_n\|_{\PP(^{k_n}\ell_p^N)}.\nonumber\end{equation}
\end{lemma}
\begin{proof}
First, we prove this lemma for two polynomials $P$ and $Q$ of degrees $k$ and $l$. We may suppose that $P$ depends on the first $r$ variables and $Q$ on the last $N-r$ ones. Given $z\in \ell_p^N$, we can write $z=x+y$, where $x$ and $y$ are the projections of $z$ on the first $r$ and the last $N-r$ coordinates respectively. We then have
\[
|P(z)Q(z)|=|P(x)Q(y)| \le  \|P\|_{\PP(^k\ell_p^N)}\, \|Q\|_{\PP(^l\ell_p^N)}\, \|x\|_p^k\, \|y\|_p^l.
\]
Since $\|z\|_p^p=\|x\|_p^p+\|y\|_p^p$, we can estimate the norm of $PQ$ as follows
\begin{align*}
\|PQ\|_{\PP(^{k+l}\ell_p^N)} = & \sup_{\|z\|_p=1} |P(z) Q(z)| \\
\le & \sup_{|a|^p+|b|^p=1} |a|^k \, |b|^l \, \|P\|_{\PP(^k\ell_p^N)} \, \|Q\|_{\PP(^l\ell_p^N)} \\
= & \sqrt[p]{\frac{k^k\, l^l}{(k+l)^{(k+l)}}} \, \|P\|_{\PP(^k\ell_p^N)} \, \|Q\|_{\PP(^l\ell_p^N)},
\end{align*}
the last equality being a simple application of Lagrange multipliers. In order to see that this inequality is actually an equality, take $x_0$ and $y_0$ norm-one vectors where $P$ and $Q$ respectively attain their norms, each with nonzero entries only in the coordinates in which the corresponding polynomial depends. If we define $$z_0=\sqrt[p]{\frac k{k+l}} \  x_0 + \sqrt[p]{\frac l{k+l}} \  y_0,$$ then $z_0$ is a norm one vector which satisfies
\[
|P(z_0)Q(z_0)|=\sqrt[p]{\frac{k^k\, l^l}{(k+l)^{(k+l)}}} \, \|P\|_{\PP(^k\ell_p^N)} \, \|Q\|_{\PP(^l\ell_p^N)}.
\]

We prove the general statement by induction on $n$. We assume the result is valid $n-1$ polynomials and we know that it is also valid for two. We omit the subscripts in the norms of the polynomials to simplify the notation. We then have
\begin{eqnarray*}
\left\|\prod_{i=1}^nP_i \right\| & = & \sqrt[p]{\frac{k_n^{k_n} \, \left(\sum_{i=1}^{n-1}k_i\right)^{\sum_{i=1}^{n-1}k_i}}{\left(\sum_{i=1}^{n}k_i\right)^{\sum_{i=1}^{n}k_i}}} \, \left\|\prod_{i=1}^{n-1}P_i \right\| \, \|P_n\| \\ \\
& = & \sqrt[p]{\frac{k_n^{k_n}\left(\sum_{i=1}^{n-1}k_i\right)^{\sum_{i=1}^{n-1}k_i}}{\left(\sum_{i=1}^{n}k_i\right)^{\sum_{i=1}^{n}k_i}}} \, \sqrt[p]{\frac{\prod_{i=1}^{n-1} {k_i^{k_i}}}{\left(\sum_{i=1}^{n-1}k_i\right)^{\sum_{i=1}^{n-1}k_i}}} \, \left(\prod_{i=1}^{n-1}{\|P_i \|}\right) \, \|P_n\| \\ \\
& = & \sqrt[p]{\frac{\prod_{i=1}^{n} {k_i^{k_i}}}{\left(\sum_{i=1}^{n}k_i\right)^{\sum_{i=1}^{n}k_i}}} \, \prod_{i=1}^{n}{\|P_i \|}. \qquad\qedhere
\end{eqnarray*}
\end{proof}
Now we are ready to prove our main result.
\begin{theorem}\label{teoprincipal}
Let $P_1,\dots,P_n$ be homogeneous polynomials of degrees $k_1,\dots,k_n$ respectively on $\ell_p^N$,  $1\le p\le 2$. If $\mathbf k=k_1+\cdots+k_n$, then we have
\begin{equation}\label{eq-principal0}
\|P_1\cdots P_n\|_{\PP(^{\mathbf k}\ell^N_p)} \ge \sqrt[p]{\frac{\prod_{i=1}^{n} {k_i^{k_i}}}{\mathbf k^{\mathbf k}}} \, \|P_1\|_{\PP(^{k_1}\ell_p^N)} \cdots \|P_n\|_{\PP(^{k_n}\ell_p^N)}.
\end{equation}
The constant is optimal provided that $N \ge n$.
\end{theorem}
\begin{proof} We first prove the inequality for two polynomials: we take homogeneous polynomials $P$ and $Q$ of degree $k$ and $l$. If $k=l$, the result follows from Remark~\ref{mismogrado-remark}. Let us suppose $k> l$. Moving to $\ell_p^{N+1}$ if necessary, we take a norm one polynomial $S$, of degree $d=k-l$, depending on different variables than the polynomials $P$ and $Q$. An example of such a polynomial is $(e'_{N+1})^d$. In the following, we identify $\ell_p^{N}$ with a subspace of $\ell_p^{N+1}$ in the natural way. We use Lemma \ref{distintasvars} for equalities \eqref{eq-paso1} and \eqref{eq-paso3}, and inequality \eqref{eq-gradosigualesp<2} for inequality \eqref{eq-paso2} to obtain:
\begin{eqnarray}
\|PQ\|_{\PP(^{k+l}\ell_p^N)} &=& \|PQ\|_{\PP(^{k+l}\ell_p^{N+1})} \, \|S\|_{\PP(^d\ell_p^{N+1})} \nonumber \\
&=&\sqrt[p]{\frac{((k+l)+d)^{(k+l)+d}}{(k+l)^{(k+l)}d^d}} \, \|PQS\|_{\PP(^{2k}\ell_p^{N+1})} \label{eq-paso1}\\
&\ge & \sqrt[p]{\frac{((k+l)+d)^{(k+l)+d}}{(k+l)^{(k+l)}d^d}} \, \frac{1}{4^{k/p}} \, \|P\|_{\PP(^{k}\ell_p^{N+1})} \, \|QS\|_{\PP(^{k}\ell_p^{N+1})} \label{eq-paso2}\\
&=& \sqrt[p]{\frac{(2k)^{2k}}{(k+l)^{(k+l)}\, d^d \, 4^k}} \, \sqrt[p]{\frac{l^ld^d}{k^k}} \, \|P\|_{\PP(^{k}\ell_p^{N+1})} \, \|Q\|_{\PP(^{l}\ell_p^{N+1})} \, \|S\|_{\PP(^d\ell_p^{N+1})} \label{eq-paso3}\\
&=& \sqrt[p]{\frac{k^kl^l}{(k+l)^{k+l}}} \, \|P\|_{\PP(^{k}\ell_p^N)} \, \|Q\|_{\PP(^{l}\ell_p^N)}. \nonumber
\end{eqnarray}
The proof of the general case continues by induction on $n$ as in the previous lemma.

To see that the constant is optimal whenever $N\ge n$, consider for each $i=1,\dots, n$ the polynomial $P_i=(e_i')^{k_i}$. From Lemma \ref{distintasvars}  we  obtain  equality in \eqref{eq-principal0}.
\end{proof}
Theorem~\ref{teoprincipal} holds also for polynomials on $\ell_p$. This is a consequence of the following: if $P\in \PP(^k\ell_p)$ then \begin{eqnarray*}\|P\|_{\PP(^k\ell_p)} = \lim_{N\rightarrow \infty}\|P \circ i_N \|_{\PP(^k\ell_p^N)} \end{eqnarray*}
where $i_N$ is the canonical inclusion of $\ell_p^N$ in $\ell_p$. The proof of this fact is rather standard. Anyway, in the next section we will show that Theorem~\ref{teoprincipal} holds for spaces $L_p(\mu)$, which comprises $\ell_p$ as a particular case.

\section{ Spaces  $L_p$ and Schatten classes }

In this section we show that the results obtained for $\ell_p$ can be extended to spaces $L_p(\Omega,\mu)$ and to the Schatten classes $\mathcal S_p$ for $1\le p\le 2$. We will sometimes omit parts of the proofs which are very similar to those in the previous section.

Let $(\Omega, \mu)$ be a measure space. From now on, the notation $\Omega=A_1 \sqcup \ldots \sqcup A_n$ will mean that it is possible to decompose the set $\Omega$ as the union of measurable subsets $\{A_i\}_{1 \le i \le n}$, such that $\mu(A_i) >0$ for $1\le i \le n$, and $\mu(A_i \cap A_j)=0$ for all $1\le i < j\le n$. Next lemma is the analogue to Lemma~\ref{distintasvars} for $L_p$ spaces.

\begin{lemma}\label{distintasvars2} Let $P,Q:L_p(\Omega,\mu) \to \mathbb C$ be homogeneous polynomials of degree $k$ and $l$ respectively. Suppose that $\Omega=A_1\sqcup A_2$, and that $P(f)=P(f\X_{A_1})$ and $Q(f)=Q(f\X_{A_2})$. Then we have
\[
\|PQ\|_{\PP(^{k+l} L_p(\Omega,\mu))}=\sqrt[p]{\frac{k^k \, l^l}{(k+l)^{(k+l)}}} \, \|P\|_{\PP(^kL_p(\Omega,\mu))} \, \|Q\|_{\PP(^lL_p(\Omega,\mu))}.
\]
\end{lemma}
\begin{proof}
Given $f\in L_p(\Omega,\mu)$ we write it as $f=f\X_{A_1}+f\X_{A_2}$ and then
\begin{align*}
|P(f)Q(f)|= & |P(f\X_{A_1})Q(f\X_{A_2})| \\
\le & \|P\|_{\PP(^kL_p(\Omega,\mu))} \, \|Q\|_{\PP(^lL_p(\Omega,\mu))} \|f\X_{A_1}\|_p^k \, \|f\X_{A_2}\|_p^l.
\end{align*}
Since $\|f\|_p^p=\|f\X_{A_1}\|_p^p +\|f\X_{A_2}\|_p^p$ the proof continues as in Lemma~\ref{distintasvars}.
\end{proof}
Combining this lemma with the fact that $D_n(L_p(\mu)) = n^{|1/p-1/2|}$ we obtain the next result.
\begin{theorem}\label{teoprincipal2}
Let $P_1,\dots,P_n$ be homogeneous polynomials of degrees $k_1, \dots, k_n$ respectively on $L_p(\Omega,\mu)$,  $1\le p\le 2$. If $\mathbf k=k_1+\cdots+k_n$, then we have
\begin{equation}\label{eq-principal2}
\|P_1\cdots P_n\|_{\PP(^{\mathbf k}L_p(\Omega,\mu))} \ge \sqrt[p]{\frac{\prod_{i=1}^{n} {k_i^{k_i}}} {\mathbf k^{\mathbf k}}} \, \|P_1\|_{\PP(^{k_1}L_p(\Omega,\mu))} \cdots \|P_n\|_{\PP(^{k_n}L_p(\Omega,\mu))}.\nonumber
\end{equation}
If $\Omega$ admits a decomposition as $\Omega=A_1\sqcup \ldots \sqcup A_n$, then the constant is optimal.
\end{theorem}
\begin{proof} We prove the result for two polynomials. Let $P$ and $Q$ be homogeneous polynomials  of degree $k$ and $l$. If $k=l$, the result follows from Remark \ref{mismogrado-remark}. Then, we can suppose $k>l$. Let us define an auxiliary measure space $(\Omega', \mu')$ by adding an additional point $\{c\}$ to $\Omega$. The measure $\mu'$ in $\Omega'$ is given by $\mu'(U)=\mu(U)$ if $U\subseteq \Omega$, and $\mu'(U)=\mu(U\cap \Omega)+1$ whenever $c\in U$. It is clear that we have $\Omega'= \Omega \sqcup \{c\}$. Let us consider the polynomials $P', Q'$ and $S$ of degree $k$, $l$ and $d=k-l$ respectively, defined on $L_p(\Omega',\mu')$ by $P'(f)=P(f|_\Omega)$, $Q'(f)=Q(f|_\Omega)$ and $S(f)=(f(c))^d$. Observe that $\|S\|_{\PP(^dL_p(\Omega',\mu'))}=1$. The polynomials $P'Q'$ and $S$ are in the conditions of Lemma~\ref{distintasvars2}. Proceeding as in the proof of Theorem~\ref{teoprincipal}, we have
\begin{align*}
\|PQ\|_{\PP(^{k+l}L_p(\Omega,\mu))} = & \|P'Q'\|_{\PP(^{k+l}L_p(\Omega',\mu'))} \, \|S\|_{\PP(^dL_p(\Omega',\mu'))} \\
= & \sqrt[p]{\frac{((k+l)+d)^{(k+l)+d}}{(k+l)^{(k+l)}d^d}}  \ \|P'Q'S\|_{\PP(^{2k}L_p(\Omega',\mu'))} \\
\ge & \sqrt[p]{\frac{((k+l)+d)^{(k+l)+d}}{(k+l)^{(k+l)}d^d}} \, \frac{1}{4^{k/p}} \, \|P'\|_{\PP(^{k}L_p(\Omega',\mu'))} \, \|Q'S\|_{\PP(^{k}L_p(\Omega',\mu'))} \\
= & \sqrt[p]{\frac{k^k \, l^l}{(k+l)^{(k+l)}}} \, \|P'\|_{\PP(^{k}L_p(\Omega',\mu'))} \, \|Q'\|_{\PP(^{l}L_p(\Omega',\mu'))} \, \|S\|_{\PP(^dL_p(\Omega',\mu'))}\\
= & \sqrt[p]{\frac{k^k \, l^l}{(k+l)^{k+l}}} \, \|P\|_{\PP(^{k}L_p(\Omega,\mu))} \, \|Q\|_{\PP(^{l}L_p(\Omega,\mu))}.
\end{align*}
The general case follows by induction exactly as in the proof of Lemma \ref{distintasvars}, and the optimality of the constant is  analogous to that of Theorem~\ref{teoprincipal}.
\end{proof}
Now we show how the previous proofs can be adapted to obtain the corresponding results for the Schatten classes.
Let $P_1,\dots,P_n:\mathcal S_p(H)\to \mathbb C$ be $k-$homogeneous polynomials on $\mathcal S_p=\mathcal S_p(H)$, the $p$-Schatten class of operators on the Hilbert space $H$. In Corollary 2.10 of \cite{T}, Tomczak-Jaegermann proved that $D_n(\mathcal S_p) \leq n^{|1/p-1/2|}$. Then, by Lemma \ref{mismogrado}, we have
\begin{equation}\label{schatteniguales}\|P_1\cdots P_n\|_{\PP(^{nk}\mathcal S_p(H)	)} \ge \frac 1 {n^{nk/p}} \, \|P_1\|_{\PP(^k\mathcal S_p(H))}\cdots \|P_n\|_{\PP(^k\mathcal S_p(H))}.\nonumber
\end{equation}
Suppose that $H=H_1\oplus H_2$ (an orthogonal sum) and let $\pi_1,\pi_2:H\rightarrow H$  be the orthogonal projections onto $H_1$ and $H_2$ respectively.
If the homogeneous polynomials $P,Q:\mathcal S_p(H) \to \mathbb C$ satisfy $P(s)= P(\pi_1\circ s\circ \pi_1)$ and $Q(s)= Q(\pi_2\circ s\circ \pi_2)$ for all $s\in \mathcal S_p$, we can think of $P$ and $Q$ as depending on different variables. Moreover, for each $s\in \mathcal S_p(H)$, it is rather standard to see that
\begin{equation}\label{eq-normas a la p}\|\pi_1\circ s\circ \pi_1\|_{\mathcal S_p}^p +\|\pi_2\circ s\circ \pi_2\|_{\mathcal S_p}^p=\|\pi_1\circ s\circ \pi_1 + \pi_2\circ s\circ \pi_2\|_{\mathcal S_p}^p.
\end{equation}
 Also, we have
$$\pi_1\circ s\circ \pi_1 + \pi_2\circ s\circ \pi_2 = \frac12  \Big( s + (\pi_1 - \pi_2)\circ s\circ (\pi_1 - \pi_2) \Big).
$$
By the ideal property of Schatten norms, the last operator has norm (in $\mathcal S_p$) not greater than $\|s\|_{\mathcal S_p}$. We then have
\begin{equation}\label{eq-normas a la p con pi} \|\pi_1\circ s\circ \pi_1\|_{\mathcal S_p}^p +\|\pi_2\circ s\circ \pi_2\|_{\mathcal S_p}^p \le \|s\|_{\mathcal S_p}^p.\end{equation}
Now, with \eqref{eq-normas a la p} and \eqref{eq-normas a la p con pi} at hand, we  can follow the proof of Lemma~\ref{distintasvars} to obtain the analogous result for Schatten classes.

Finally, the trick of adding a variable in Theorem~\ref{teoprincipal} or a singleton in Theorem~\ref{teoprincipal2} can be performed for Schatten classes just taking the orthogonal sum of $H$ with a (one dimensional) Hilbert space. As a consequence, mimicking the proof of Theorem~\ref{teoprincipal} we obtain the following.
\begin{theorem}\label{teoprincipal3}
Let $P_1,\dots,P_n$ be polynomials of degrees $k_1,\dots,k_n$ respectively on $\mathcal S_p(H)$ with $1\le p\le 2$. If $\mathbf k=k_1+\cdots+k_n$, then we have
\begin{equation}\label{eq-principal} \|P_1\cdots P_n\|_{\PP(^{\mathbf k}\mathcal S_p(H))} \ge \sqrt[p]{\frac{\prod_{i=1}^{n} {k_i^{k_i}}} {\mathbf k^{\mathbf k}}}\,  \|P_1\|_{\PP(^{k_1}\mathcal S_p(H))} \cdots \|P_n\|_{\PP(^{k_n}\mathcal S_p(H))}.\nonumber
\end{equation}
The constant is optimal provided that $dim(H)\ge n$.
\end{theorem}

\section{Remarks on the case $p>2$}

We end this note with some comment on the constant for $p>2$. If $X$ is a Banach space, let $M(X, k_1, \ldots, k_n)$ be the largest value of $M$ such that
\begin{equation*}\label{best}
\Vert P_1 \cdots P_n \Vert_{\PP(^{\kkk} X)} \ge M \ \Vert P_1\Vert_{\PP(^{k_1}X)}\cdots \Vert P_n\Vert_{\PP(^{k_n}X)}
\end{equation*}
for any set of homogeneous polynomials $P_1, \ldots, P_n$ on $X$, of degrees $k_1, \ldots, k_n$ respectively. From \cite{BST}, \cite{P} and Theorem \ref{teoprincipal} we know that
\[
M(\ell_p^N, k_1, \ldots, k_n)=\sqrt[p]{\frac{\prod_{i=1}^{n}{k_i^{k_i}}} {\kkk^{\mathbf k}} },
\]
provided that $1\le p \le 2$ and $ N \ge n$.
In \cite[Proposition 8]{RS}, the authors show that the best constant for products of linear functionals on an infinite dimensional Banach space is worse than the corresponding one for Hilbert spaces. In our notation, they show that $$M(\ell_2, 1, \ldots, 1) \ge M(X, 1, \ldots, 1)$$ for every infinite dimensional Banach space $X$.
Next theorem, together with Theorems \ref{teoprincipal} and \ref{teoprincipal2}, show that the same holds for products of homogeneous polynomials in $\ell_p^N$ and $L_p$ spaces, provided that the dimension is greater than or equal to the number of factors. That is, the constant for Hilbert spaces is better than the constant of any other $L_p$ space  for homogeneous polynomials of any degree, even in the finite dimensional setting.
\begin{theorem} For $ N\ge n$ and $2\le p\le \infty$, we have
\[
M(\ell_2^{N}, k_1, \ldots, k_n) \ge M(\ell_p^{N}, k_1, \ldots, k_n) \ge   \big({{n}^{k_1+ \cdots + k_n}}\big)^{\frac 1 p - \frac 1 2}  \  M(\ell_2^{N}, k_1, \ldots, k_n).
\]
The same holds for $L_p(\Omega,\mu)$ whenever $\Omega$ admits a decomposition as in Theorem~\ref{teoprincipal2}.
\end{theorem}
\begin{proof} The second inequality is a direct consequence of Lemma~\ref{mismogrado}, so let us show the first one. Consider the linear forms on $\ell_p^{N}$ defined by the vectors
$$
g_j=\left(1,e^{\frac{2\pi i j}{N}},e^{\frac{2\pi i 2j}{N}},e^{\frac{2\pi i 3j}{N}}, \ldots ,e^{\frac{2\pi i (N-1)j}{N}}\right) \mbox{ for } j=1, \ldots ,n.
$$
These are orthogonal vectors in $\ell_2^N$. We can choose an orthogonal coordinate system such that the $g_i$'s  depend on different variables (we are in $\ell_2^N$). So by Lemma~\ref{distintasvars}, inequality ~\eqref{problema} holds as an equality with the constant for Hilbert spaces given in \eqref{ctel2}:
$$
M(\ell_2^{N}, k_1, \ldots, k_n) \ \|g_1^{k_1}\|_{\PP(^{k_1}\ell_2^N)} \ldots \|g_n^{k_n}\|_{\PP(^{k_n}\ell_2^N)} =  \, \|g_1^{k_1} \ldots g_n^{k_n}\|_{\PP(^{\kkk}\ell_2^N)}.
$$
For products of orthogonal linear forms this equality was observed in~\cite{ADR} and for the general case (with arbitrary powers) in Remark 4.2 of~\cite{P}.

On the other hand, we have $\|g_j^{k_j}\|_{\PP(^{k_j}\ell^N_2)}=(N^{1/2})^{k_j}$, and  $\|g_l^{k_j}\|_{\PP(^{k_j}\ell_p^N)}=(N^{1-\frac{1}{p}})^{k_j}$. Combining all this we obtain the following:
\begin{eqnarray*} M(\ell_2^{N}, k_1, \ldots, k_n)&=&\frac {\|g_1^{k_1} \ldots g_n^{k_n}\|_{\PP(^{\kkk}\ell_2^N)} } {\|g_1^{k_1}\|_{\PP(^{k_1}\ell_2^N)} \ldots \|g_n^{k_n}\|_{\PP(^{k_n}\ell_2^N)}}\\
&= &  \frac {\|g_1^{k_1} \ldots g_n^{k_n}\|_{\PP(^{\kkk}\ell_2^N)} } {(N^{\frac{1}{p}-\frac{1}{2}})^{k_1}\|g_1^{k_1}\|_{\PP(^{k_1}\ell_p^N)} \ldots (N^{\frac{1}{p}-\frac{1}{2}})^{k_n}\|g_n^{k_n}\|_{\PP(^{k_n}\ell_p^N)}}\\
&\ge &  \frac {\|g_1^{k_1} \ldots g_n^{k_n}\|_{\PP(^{\kkk}\ell_p^N)} N^{\left(\frac{1}{p}-\frac{1}{2}\right)\kkk }} {(N^{\frac{1}{p}-\frac{1}{2}})^{k_1}\|g_1^{k_1}\|_{\PP(^{k_1}\ell_p^N)} \ldots (N^{\frac{1}{p}-\frac{1}{2}})^{k_n}\|g_n^{k_n}\|_{\PP(^{k_n}\ell_p^N)}} \\
&= & \frac {\|g_1^{k_1} \ldots g_n^{k_n}\|_{\PP(^{\kkk}\ell_p^N)} } { \|g_1^{k_1}\|_{\PP(^{k_1}\ell_p^N)} \ldots \|g_n^{k_n}\|_{\PP(^{k_n}\ell_p^N)}}
 \\
& \ge & M(\ell_p^{N}, k_1, \ldots, k_n).
\end{eqnarray*}

This shows the statement for $\ell_p^N$. Since the space $L_p(\Omega,\mu)$, with our assumptions on~$\Omega$, contains a 1-complemented copy of $\ell_p^n$, the statement for $L_p(\Omega,\mu)$ readily  follows. \end{proof}

\end{document}